\pdfoutput=1
\documentclass{article}
\usepackage[utf8]{inputenc}

\usepackage{amsthm,amsfonts,amssymb,amsmath,epsf, verbatim}
\usepackage[utf8]{inputenc}

\def\rad{{\rm {rad }}}

\newtheorem{theorem}{Theorem}

\newtheorem{proposition}[theorem]{Proposition}

\newtheorem{conjecture}[theorem]{Conjecture}

\title{A Note on a Result of Makowski}
\author{Luis H. Gallardo\footnote{Univ. Brest, Laboratoire de Math\'ematiques de Bretagne Atlantique, Brest, France, luishge11@gmail.com}, Joshua Zelinsky\footnote{Department of Mathematics, Hopkins School, New Haven, CT, USA,  zelinsky@gmail.com}}
\date{}

\begin{document}

\maketitle
\begin{abstract}
In this note, we fix a gap in a proof of the first author that 28 is the only even perfect number which is the sum of two perfect cubes. We also discuss the situation for higher powers.
\end{abstract}

In this note, we fix a gap in a proof of the first author. Makowksi  \cite{Makowski} proved that 28 is the only even perfect number of the form $x^3+1$. Motivated by this, the first author \cite{Gallardo} stated that if $n$ is an even perfect number, and $n=x^3 + a^3$ for some positive integers $x$ and $a$, then we must have $n=28$.\footnote{That the cubes are positive is needed here. A referee of this paper pointed out that the prior paper also implicitly used that the cubes were positive without stating explicitly. If one does not insist on the cubes being positive, one has other examples such as $8128=  28^3 - 24^3$.} However, there is a step in that proof that has a gap. In particular, if $n$ is an even perfect number, then we may factor $n$ as \begin{equation}\label{basic factorization}
n= x^3 + a^3 = (x+a)(x^2 - ax+a^2)
\end{equation}
and, with some work, reach a contradiction. However, the contradiction reached  required assuming that $(x+a, x^2 - ax + a^2) =1$, which is not necessarily the case. The main purpose of this note to is to address this situation. We then discuss representing even perfect numbers as the sum of two higher powers.

We recall the Euclid-Euler theorem, which states that $n$ is an even perfect number if and only if there is a prime $2^p-1$ such that 
\begin{equation}\label{Euclid-Euler theorem}
n = (2^p-1)2^{p-1}.
\end{equation}

Note that if $2^p-1$ is prime, then $p$ is prime but the converse does not follow. 

We will, following  the first author's prior work \cite{Gallardo}, factor $n$ as in Equation \ref{basic factorization}. We will also assume that there is a non-trivial common factor between $x+a$ and $x^2-ax+a^2$. Note that since the only prime factor repeated in the factorization of $n$ is $2$, it must be the case that our common factor is a power of 2. 

We first want to show that $2^p-1$ is not a factor of $x+a$.  Assume that $2(2^p-1)|x+a$.  We consider then two cases. In case I, $2(2^p-1)=x+a$. In case II, $4(2^p-1)|x+a$. \\

Case I: Assume that 
\begin{equation} 2(2^p-1)=x+a \label{2(2p -1) = x+a}.\end{equation}
Then we have
\begin{equation} 2^{p-2} = x^2-ax+a^2 .
\label{2 (p-2) = x2 -ax -a2}
\end{equation}

Then we may combine equation Equation \ref{2(2p -1) = x+a} and Equation \ref{2 (p-2) = x2 -ax -a2} to get that

$$\frac{x+a+2}{8} = x^2 -ax +a^2$$

which has a negative discriminant if $a \geq 1$, and thus has no solution in positive integers. We may thus assume that we are in Case II where $$4(2^p-1)|x+a.$$ Thus we have that $2^p -1 \leq \frac{x+a}{4}$,
and so 

$$x^3+a^3 = (2^p-1)2^{p-1} \leq (\frac{x+a}{4}) (\frac{x+a}{8} + \frac{1}{2}) $$

which is a contradiction when $x \geq 1$ and $a \geq 1$.

We thus may assume that $2^p-1|x^2-ax+x^2$. For the remainder of the proof, we will assume that \begin{equation} x+a= 2^h \label{x+a=2 to h}
\end{equation}
and 
\begin{equation} x^2 - ax +a^2 = (2^p-1)2^{p-1-h} \label{x2 -ax +a2}.
\end{equation}
We have then
\begin{equation}
    2^{2h} = (x+a)^2 = x^2 +2ax + a^2 > x^2 -ax +a^2 \label{x+a squared inequality}.
\end{equation}
Thus, \begin{equation} 2^{2h} > (2^p-1)2^{p-1-h}\end{equation}
which implies that $2^{3h} > 2^{2p-1} - 2^{p-1}$ and hence that 
\begin{equation}3h \geq 2p-1. \end{equation}

We also have that

\begin{equation} 2^{2h} = (x+a)^2 \leq 4(x^2-ax +a^2) = 4(2^p-1)2^{p-1-h} < 2^{2p+1}2^{-h}. \label{2h upper bound}
\end{equation}
We have from the above chain of inequalities that $3h \leq 2p$. Thus, we have either $3h = 2p-1$ or $3h = 2p$. If $3h=2p$, then $p=3$ and and thus we have $n=28 = 3^3 +1^3$. We thus need only to consider the situation where $3h=2p-1$.\\

Note that we must have $p \equiv 2$ (mod 3). We set $p=3q+2$. \\
So \begin{equation} h= \frac{6q+3}{3} = 2q+1.
\end{equation}
We then may apply Equations \ref{x+a=2 to h} and \ref{x2 -ax +a2} to get that 
\begin{equation} x+a = 2^{2q+1}= 2(2^{2q}) \label{x+a in terms of q}
\end{equation}
and 
\begin{equation} x^2 -ax+a^2 = (2^{3q+2}-1)(2^{3q+1-(2q+1)}) = 2^{4q+2} - 2^q. \label{x2 -ax + a2 in terms of q}
\end{equation}
If we set $N=2^q$, we get from Equation \ref{x+a in terms of q} and Equation \ref{x2 -ax + a2 in terms of q} that 

\begin{equation}\label{x+a= 2N squared} x+a=2N^2
\end{equation}
and 
\begin{equation}\label{x2 -ax +a2 = (2N) squared - N} x^2 -ax +a^2 = 4N^4 - N.
\end{equation}
Thus, we have 

$$x^2 -ax + a^2 = (x+a)^2 - \sqrt{\frac{x+a}{2}},$$ but this equation does not have any integer solutions. The proof is now complete.\\










We strongly suspect the following holds:

\begin{conjecture} If $n$ is an even perfect number, and $n =x^m +y^m$ where $x$ and $a$ are positive integers and $m$ is a positive integer which is at least 2, $m \geq 2$, then we must have $m=3$, and $n=28$. 
\end{conjecture}
This conjecture is motivated by three things. First, we we know it is true for $m=3$ by the earlier part of this note. Second, it is true for any even $m$, since every Mersenne prime is 3 (mod 4), and thus no even perfect number is the sum of two perfect squares. Third, for any $m \geq 5$, the set of numbers of the form $x^m + a^m$ has a very low density.   

However, trying to apply the method above to prove this seems insufficient even for the case of $m=5$. To see where this breaks down, let us assume that we have an even perfect number where $m=5$ and try to follow through the same line of logic we had above.  

Let us assume that 
\begin{equation}\label{n = x5 y5} n = x^5 + a^5\end{equation}
for positive integers, $x$ and $a$. As before, we may write 

\begin{equation}n = 2^{p-1}(2^p-1).\end{equation}
We have from Equation \ref{n = x5 y5},

\begin{equation} n= (x + a)\left(x^{4} - x^{3}a + x^{2}a^2  - xa^3 + a^{4} \right).
\end{equation}

We will write $A=x+a$, and $B=\left(x^{4} - x^{3}a + x^{2}a^2  -xa^3 + a^{4} \right)$.
 Since $x$ and $a$ are not both equal to 1, we get 
\begin{equation}
    2(x+a)^3 \leq x^5 + a^5 
\end{equation}
and so 
\begin{equation}2A^2 < B.\label{2A squared < B}\end{equation}

We thus have that $2^p-1|B$, since $2^p-1$ is prime, $2^p-1|AB$, and if we had $2^p-1|A$ we would be forced to violate Inequality \ref{2A squared < B}. We thus have a situation similar to what we had before, where

\begin{equation}A= 2^h\label{A = 2 to h},\end{equation}
and \begin{equation}B = 2^{p-1-h}(2^p-1) \label{B in terms of h}.
\end{equation}
We  have then from Inequalities \ref{2A squared < B}, \ref{A = 2 to h}, \ref{B in terms of h} that
\begin{equation} 3h \leq 2p-2. 
\end{equation}
We are thus in a situation very similar to our $m=3$ situation. But we need to get a lower bound on $h$ as well, and now we run into a problem. The best we seem to be able to do is
\begin{equation} A^4 \geq B,
\end{equation}
and this only gives us after simplifying 
\begin{equation} 3h \geq p-1.
\end{equation}
And this is too wide a range of possible $h$ values for our earlier strategy to work without some additional insight. 

We can show, subject to an explicit version of the ABC conjecture, that we can rule out even perfect numbers which are the sum of two non-trivial powers of high degree with certain properties.

\begin{conjecture}\label{Baker's Conjecture} Let $A$, $B$, $C$ be positive integers such that $A+B=C$, and $(A,B)=1$. Then \begin{equation}
\label{abcBaker2}
\max(A,B,C) < \rad(ABC)^{7/4}.
\end{equation}
\end{conjecture}

Note that Conjecture \ref{Baker's Conjecture} follows from Baker's explicit version of the ABC conjecture.\cite{LaishramShorey}

Then we have the following:

\begin{proposition}
    
\label{ABCapplied}
Assume that Conjecture 2 holds. Assume that $2^p-1$ is prime. Then there is no simultaneous solution in positive integers satisfying the following list of conditions: 
(a) $m > 29$, \\
(b) $(x+a)^{m-2} \leq x^m+a^m$,\\
(c) Both $x$ and $a$ are odd,\\
(d) $x+a = 2^h$,\\
(e) $\frac{x^m+a^m}{x+a} = (2^p-1) \cdot 2^{p-1-h}$.\\

\end{proposition}
    
\begin{proof}
    
Assume Conjecture 2 (ABC  with exponent $7/4$) holds. Let $p$ be  a prime number such that $2^p-1$ is also prime. Moreover, assume that all conditions  (a),(b),(c),(d) and (e) hold. Note that we may assume that $p>3$.\\
Put $A= x^m$, $B=a^m$, and $C=(2^p-1) \cdot 2^{p-1}$.\\
Note that  (d) and (e) imply that
 \begin{equation}
\label{zeta}
2^{p-1} \cdot (2^p-1) = x^m+a^m. 
\end{equation}

By conditions (c) and (d) we have
\begin{equation}
\label{zero}
\gcd(A,B)=1.
\end{equation}
Thus Conjecture 2 implies that
\begin{equation}
\label{uno}
(2^p-1) \cdot 2^{p-1} < \left(2 \cdot \textrm{rad}(a x) \cdot (2^p-1) \right)^{7/4}
\end{equation}
From \eqref{uno} raised to the $4$-th power we get
\begin{equation}
\label{dos}
(2^p-1)^4 \cdot 2^{4 p-4} < 2^7 \cdot  \textrm{rad}(a x)^7 \cdot (2^p-1)^7.
\end{equation}
Thus, \eqref{dos} implies
\begin{equation}
\label{tres}
2^{4 p-11} < \textrm{rad}(a x)^7 \cdot (2^p-1)^3.
\end{equation}
But since
$4 a x  \leq (x+a)^2$, and
 $\textrm{rad}(a x) \leq a x$, it follows from (d) that
\begin{equation}
\label{seis}
\textrm{rad}(a x)  \leq 2^{2 h -2}.
\end{equation}
Thus, \eqref{tres} implies
\begin{equation}
\label{siete}
2^{4 p-11} <  2^{14 h -14} \cdot (2^p - 1)^3 < 2^{14 h -14} \cdot 2^{3 p}.
\end{equation}
We deduce from \eqref{siete} that
\begin{equation}
\label{ocho}
2^{p-11} <  2^{14 h -14}.
\end{equation}

Therefore, we have
\begin{equation}
\label{diez}
14 h -14 \geq p - 9
\end{equation}
since $p$ is odd.
Inequality \eqref{diez} implies that
\begin{equation}
\label{onze}
h \geq \frac{p +5}{14}.
\end{equation}
Now, we will find an upper bound for $h$. 
Using (b), (d), and (e) we deduce that
\begin{equation}
\label{trece}
2^{h \cdot(m-2)} \leq 2^h \cdot (2^p -1) \cdot  2^{p -1 -h} \leq 2^h \cdot 2^p \cdot 2^{p -1 -h}.
\end{equation}
Inequality \eqref{trece} divided by $2^h$ becomes:
\begin{equation}
\label{catorce}
2^{h \cdot(m-3)} \leq 2^{2 p -1 -h} .
\end{equation}
From \eqref{catorce} we get
\begin{equation}
\label{quince}
2 p -1 - h \geq h \cdot(m-2) - h.
\end{equation}
In other words \eqref{quince} says that
\begin{equation}
\label{dieciseis}
2 p -1  \geq h \cdot(m-2).
\end{equation}
And \eqref{dieciseis} can be also written as
\begin{equation}
\label{diecisiete}
h \leq \frac{2 p -1}{m-2}.
\end{equation}
Putting together the lower bound \eqref{onze} of $h$ with the upper bound \eqref{diecisiete} of $h$ we get
\begin{equation}
\label{dieciocho}
\frac{p+5}{14} \leq \frac{2 p -1}{m-2},
\end{equation}
which is equivalent to
\begin{equation}
\label{diecinueve}
m- 2 \leq f(p),
\end{equation}
where
\begin{equation}
\label{veinte}
f(p) := \frac{14 \cdot (2 p-1)}{p+5} =\frac{-154}{p+5} +28 . 
\end{equation}
It follows from \eqref{diecinueve} and \eqref{veinte} that
\begin{equation}
\label{veintidos}
m \leq f(p)+2 = 30 -  \frac{154}{p+5}.
\end{equation}
Then Inequality \eqref{veintidos} contradicts condition (a) and the result follows.

\end{proof}

\section*{Acknowledgments}
The authors acknowledge anonymous Wikipedia editor at IP address 219.78.80.30 who brought the gap to the authors' attention. We also acknowledge the many improvements made by the referee's careful attention to detail.

\end{document}